\documentclass[12pt,notitlepage]{amsart}
\usepackage{latexsym,amsfonts,amssymb,amsmath,amsthm} 
\usepackage[inner=1.0in,outer=1.0in,bottom=1.0in, top=1.0in]{geometry}

\newcommand{\mz}{\ensuremath{\mathbb Z}}

\newcommand{\mh}{\ensuremath{\mathbb H}}
\newcommand{\mq}{\ensuremath{\mathbb Q}}
\newcommand{\mc}{\ensuremath{\mathbb C}}

\newcommand{\mymod}{\ensuremath{\negthickspace \negmedspace \pmod}}
\newcommand{\shortmod}{\ensuremath{\negthickspace \negthickspace \negthickspace \pmod}}

\newcommand{\half}{\ensuremath{ \frac{1}{2}}}

\newcommand{\intR}{\int_{-\infty}^{\infty}}

\newcommand{\thalf}{\tfrac12}
\newcommand{\sumstar}{\sideset{}{^*}\sum}
\newcommand{\leg}[2]{\left(\frac{#1}{#2}\right)}

\newcommand{\zL}[2]{L({#1},{#2})}

\theoremstyle{plain}		
	\newtheorem{mytheo}{Theorem}[section]
	
	\newtheorem{myprop}[mytheo]{Proposition}
	
     \newtheorem{mylemma}[mytheo]{Lemma}

\theoremstyle{remark}

\def\lam{\lambda}
\begin{document}
\title{The prime geodesic theorem}  
\author{K. Soundararajan}
\address{Stanford University\\
450 Serra Mall, Building 380\\
Stanford, CA 94305-2125}
 \email{ksound@math.stanford.edu}

\author{Matthew P. Young}
\address{Department of Mathematics \\
	  Texas A\&M University \\
	  College Station \\
	  TX 77843-3368 \\
		U.S.A.}
\email{myoung@math.tamu.edu}
\thanks{ The authors are partially supported by grants from the NSF (DMS-0500711  and DMS-0758235)}

\maketitle
\section{Introduction}

\noindent The closed geodesics on the modular surface $\Gamma \backslash {\mh}$, 
$\Gamma =PSL(2,\mz)$, are of arithmetic interest since their lengths 
correspond to regulators of real quadratic fields and these lengths appear with 
multiplicity equal to the class number.   A striking application of this connection is 
the work of Sarnak \cite{Sarnak} where he evaluates the average class number of 
real quadratic fields when the fields are ordered by the size of the regulator.   A key input 
in this application is the Selberg trace formula where the closed geodesics 
appear on the geometric side and they arise as the norms of  
hyperbolic conjugacy classes of ${\Gamma}$.

There is also a beautiful analogy between prime numbers and primitive closed geodesics on $\Gamma \backslash \mathbb{H}$.  Any matrix in a hyperbolic conjugacy class $\{P\}$ in $\Gamma$ 
is conjugate over $SL(2,\mathbb{R})$ to a matrix of the form $\begin{pmatrix} \lambda & 0 \\ 0 & \lambda^{-1} \end{pmatrix}$ with $\lambda >1$.   The trace of such a conjugacy class is 
${\rm tr}(P)=\lambda+\lambda^{-1}$, and we define its norm to be $NP = \lambda^2$.   The hyperbolic conjugacy class $\{ P\}$ is called primitive if $P$ cannot be written as $Q^j$ for some
$Q \in \Gamma$ and $j\ge 2$.   Note that every hyperbolic conjugacy class is a power 
of some primitive class.  We denote by $\{P_0\}$ a typical primitive hyperbolic conjugacy class
in $\Gamma$, and let  $\pi_{\Gamma}(x)$ count the number of primitive 
conjugacy classes with norm below $x$.  Then the prime geodesic theorem states
\begin{equation*}
 \pi_{\Gamma}(x) = 
\int_0^{x} \frac{dt}{\log{t}}  + E(x)= {\rm li}(x)+ E(x),
\end{equation*}
where $E(x)$ is an error term which is the object of our study in this paper.  
Just as the error term in the prime number theorem is related to zeros 
of the Riemann zeta-function via the explicit formula, the error term in the prime 
geodesic theorem may be related to the spectrum of the Laplacian 
on ${\Gamma }\backslash \mathbb{H}$, or equivalently to the zeros 
of the Selberg zeta-function.   Recall that the Selberg zeta function is 
built out of prime geodesics:  
\begin{equation*}
 Z(s) = \prod_{\{P_0\}} \prod_{k=0}^{\infty} (1 - (NP_0)^{-s-k}).
\end{equation*}
The Selberg zeta function is similar to the Riemann zeta function in 
many ways, but there are also crucial differences.  The non-trivial zeros of $Z(s)$ satisfy 
the analog of the Riemann hypothesis.  They may be 
written as $1/2+ i t_j$ and correspond to the 
eigenvalues $1/4 +t_j^2$ of the Laplacian.  Unlike the Riemann 
zeta function which is of order $1$, the Selberg zeta function is meromorphic 
of order $2$.   Given that the Riemann hypothesis holds 
for $Z(s)$ one may expect that the 
error term in the prime geodesic theorem satisfies $E(x)\ll x^{1/2+\epsilon}$, 
but since $Z(s)$ is of order $2$ it has many more zeros 
than $\zeta(s)$, and as an immediate consequence of the 
Selberg trace formula one only obtains that $E(x)\ll x^{3/4+\epsilon}$.  
Nevertheless it is believed that $E(x)\ll x^{\frac 12+\epsilon}$, and 
this remains an outstanding open problem.

 Iwaniec \cite{IwaniecGeodesic} was the first to break the $3/4$ barrier.  
 He showed that $E(x) \ll x^{35/48+ \varepsilon}$, 
 and remarked (see \cite{IwaniecNonholo}, Section 7) that the 
 bound $E(x) \ll x^{2/3 + \varepsilon}$ would 
 follow from the Generalized Lindel{\" o}f Hypothesis for Dirichlet $L$-functions.  
Iwaniec's proof uses the Selberg trace formula, the Kuznetsov formula, 
and Burgess' estimate for character sums.  To go 
between the Selberg and Kuznetsov trace formulas, he 
required information about the size of symmetric-square $L$-functions on the critical line.  
This step in Iwaniec's proof was strengthened by 
Luo and Sarnak \cite{LuoSarnak}, using work of Hoffstein and Lockhart \cite{HL}, 
who thus obtained $E(x) \ll x^{7/10 + \varepsilon}$.   
By refining the step in Iwaniec's proof that uses the Burgess estimate, 
Cai \cite{Cai} improved the Luo-Sarnak bound to get $E(x) \ll x^{71/102 + \varepsilon}$.   
We shall improve these estimates and establish that $E(x)\ll x^{25/36+\varepsilon}$.  
Before stating our result, we note that as in prime number theory it is convenient 
to consider 
$\Psi_{\Gamma}(x) =\sum_{NP \le x} \Lambda(P)$ where the sum is over all 
hyperbolic conjugacy classes $\{ P\}$ and if $\{P\}$ is a power of 
the primitive hyperbolic class $\{P_0\}$ then $\Lambda(P) = \log NP_0$.  
By partial summation we may easily pass between asymptotic 
formulae for $\Psi_\Gamma(x)$ and $\pi_{\Gamma}(x)$.  

\begin{mytheo}
\label{thm:mainthm}
We have
$$
\pi_{\Gamma}(x) =  {\rm li}(x) + O(x^{25/36+\varepsilon}), 
\qquad \text{and} \qquad \Psi_{\Gamma}(x) = x+ O(x^{25/36+\varepsilon}).
$$ 
\end{mytheo}

Our approach to this Theorem is a little different 
from the previous methods in that we emphasize more 
strongly the connection between prime geodesics 
and, via the class number formula, Dirichlet $L$-functions.   There are 
two parts to the proof: first, using the Selberg trace formula 
we handle  smoothed sums over 
prime geodesics, and second, using Dirichlet $L$-functions 
we handle prime geodesics in short intervals.  Knowledge about prime 
geodesics in short intervals allows us to ``unsmooth" the information 
obtained via the trace formula.  This approach gives us a 
relatively simple and self-contained way of broaching the $3/4$-barrier, 
and, as we indicate in \S 3, we can recover Iwaniec's original result $E(x)\ll x^{35/48+\epsilon}$ 
with a very different proof which ``leaves the theory 
of automorphic functions on a side" --- Iwaniec wondered about such 
a possibility in \cite{IwaniecGeodesic}, page 142.   
The proof of Theorem 1.1 follows upon inputing the work of Luo and Sarnak 
mentioned earlier, and the more recent work of Conrey and Iwaniec \cite{ConreyIwaniec} 
who obtained improvements over the Burgess bound for quadratic Dirichlet $L$-functions.

We also consider here the prime geodesic theorem in short intervals.  Note that 
the trace of a hyperbolic conjugacy class must be a natural number $n>2$, and 
its norm is $((n+\sqrt{n^2-4})/2)^2=n^2 -2 +O(n^{-1})$.  Thus unlike prime numbers 
which have an average gap of $\log x$, the norms of prime geodesics are widely 
spaced with each possible norm appearing with high multiplicity.   Clearly the 
asymptotic $\pi_{\Gamma}(x+h) -\pi_{\Gamma}(x) \sim h/\log x$ 
cannot always hold if $h \le \sqrt{x}$, while it is expected that $\pi(x+h) -\pi(x) 
\sim h/\log x$ holds throughout the range $x^{\half+\epsilon} \le h\le x$.   Proving a 
conjecture of Iwaniec, Bykovskii \cite{B} established that if $x^{1/2+\epsilon} 
\le h\le x$ we have $\pi_{\Gamma}(x+h)-\pi_{\Gamma}(x) \sim h/\log x$.  
Bykovskii's interesting work uses zero density results for Dirichlet $L$-functions 
and he remarked that the method would give the asymptotic formula 
in the range $x\ge h\ge \sqrt{x} \exp(c\sqrt{\log x\log \log x})$ for some constant $c>0$.  
We show that, if one assumes the Generalized Riemann Hypothesis for 
quadratic Dirichlet $L$-functions, then the prime geodesic theorem 
holds in short intervals $[x,x+h]$ provided $x\ge h\ge \sqrt{x} (\log x)^{2+\epsilon}$.  
   
\begin{mytheo}
\label{thm:shortintro}
 Assume the Generalized Riemann Hypothesis for quadratic Dirichlet $L$-functions. 
 In the range $x^{\half} (\log x)^{2+\varepsilon} \leq h \leq x$ we have
\begin{equation*}
 \pi_{\Gamma}(x+h) - \pi_{\Gamma}(x) \sim \frac{h}{\log x}, \qquad 
 \text{and} \qquad \Psi_{\Gamma}(x+h)-\Psi_{\Gamma}(x)\sim h.
\end{equation*}
\end{mytheo}

As mentioned before, the Selberg trace formula connects prime geodesics with the 
eigenvalues of the Laplacian in much the same way as the explicit formula 
weds primes and zeros of the Riemann zeta function.  The Kuznetsov formula 
connects eigenvalues of the Laplacian with Kloosterman sums.  Combining 
the two we see that there is a connection between prime geodesics and Kloosterman 
sums, and this is implicit in the works of Iwaniec, and Luo and Sarnak mentioned 
earlier.  We make this connection explicit and give a direct way of going from 
prime geodesics to Kloosterman sums.

 \begin{mytheo}
\label{thm:L1sum}
Let $w$ be a smooth, even, compactly-supported function with $w(t)=0$ for $t \in [-2,2]$.   
Let 
${\widehat w}$ denote its Fourier transform ${\widehat w}(\xi)=\int_{-\infty}^{\infty} w(x)e^{-2\pi ix\xi}dx$.   
Then, with $P$ running over all hyperbolic conjugacy classes, 
\begin{equation*}
\label{eq:L1avg}
 \sum_{\{P\}} \Lambda(P) \frac{w({\rm tr}(P))}{\sqrt{{\rm tr}(P)^2-4}} = \zeta(2) \sum_{q=1}^{\infty} q^{-2} \sum_{l \in \mz} S(l^2, 1;q) \widehat{w}
 \Big(\frac lq\Big).
\end{equation*}
\end{mytheo}

While this formula is quite pretty, the convergence of the sums on the right hand side is a 
little delicate.  So the exact formula may not be of use in applications, but one should be able to 
work out quantitative approximate formulae which may be useful.   

\section{Descriptions of the prime geodesics}
\label{section:Initial}

\noindent   In this section we collect together descriptions of 
the prime geodesics for $\Gamma$, and describe some preliminary 
results which will be used in the proofs of the main theorems.    
Many of the results given below are well known to the 
experts, but we have opted to recall and sketch them briefly for the 
convenience of the reader.  
Throughout the paper we shall use the notation $X= \sqrt{x}+1/\sqrt{x}$.  
The condition $NP \le x$ is then equivalent to $\text{tr}(P) \le X$.

 Sarnak (\cite{Sarnak}, Proposition 4) showed that  primitive hyperbolic conjugacy classes correspond to equivalence classes of primitive indefinite
 binary quadratic forms.  Sarnak's bijection is as follows:  Given a primitive 
 binary quadratic form $ax^2 + bxy+ cy^2$ (primitive means that g.c.d.$(a, b,c)=1$) of discriminant $d$, the automorphs of this 
 form are $\pm P(t,u)$ where   
\begin{equation*}
 P(t,u) = \begin{pmatrix}
           \half(t-bu) & -cu \\ au & \half(t+bu)
          \end{pmatrix},
\end{equation*}
with $t^2 - du^2 = 4$ being a solution to the Pell equation.  
For non-zero $u$, $P(t,u)$ is hyperbolic with norm $(t + u\sqrt{d})^2$ and trace $t$.  
If $(t_d,u_d)$ denotes the fundamental solution to the Pell equation, 
then $P(t_d, u_d)$ is a primitive hyperbolic matrix with trace $t_d$ and norm $\epsilon_d^2$, 
and the other $P(t,u)$ are powers of $P(t_d,u_d)$.  
Sarnak's correspondence 
maps the primitive quadratic form $ax^2 + bxy + cy^2$ of discriminant $d$ to the primitive hyperbolic element $P(t_d, u_d)$.  Thus for every discriminant $d$ we 
see that there are $h(d)$, the class number, primitive hyperbolic conjugacy classes 
and these all have the same trace $t_d$ and norm $\epsilon_d^2$.

Now every hyperbolic conjugacy class $\{P\}$ may  be expressed as $\{P_0^j\}$ for a natural number $j$ and a primitive hyperbolic conjugacy class $\{P_0\}$.  From the correspondence 
described above, we see that with $d$ denoting a discriminant,
$$ 
\Psi_{\Gamma}(x) = \sum_{2<t \le X} \sum_{\substack{d \\ t^2 -du^2 =4 }} h(d) \log (\epsilon_d^2).  
$$  
The class number formula $h(d) \log \epsilon_d = \sqrt{d} L(1,\chi_d)$, where 
$\chi_d$ is the (not necessarily primitive) real character associated to the discriminant $d$, 
allows us to write 
\begin{equation} 
\label{eq:L1}
\Psi_{\Gamma}(x) = 2 \sum_{2< t\le X} \sum_{du^2=t^2 -4} \sqrt{d}L(1,\chi_d).
\end{equation} 

We now describe how the inner sum in \eqref{eq:L1} may be expressed as 
$\sqrt{t^2 -4} L(1,t^2-4)$ for a certain ``natural" Dirichlet series $L(s,\delta)$ defined 
for all discriminants $\delta$.   The Dirichlet series $L(s,\delta)$ that we need 
appears in the work of Bykovskii \cite{B}, and in a different context  may be 
found in the work of Zagier \cite{Zagier}.   For a discriminant $\delta$, and 
a natural number $q$ we define 
 \begin{equation*}
\rho_q(\delta) = \#\{x \mymod{2q} : x^2 \equiv \delta \mymod{4q} \},
\end{equation*} 
and 
\begin{equation*}
 \lambda_q(\delta) = \sum_{q_1^2 q_2 q_3 = q} \mu(q_2) \rho_{q_3}(\delta), 
 \qquad \text{so that  } \qquad \rho_q(\delta)= \sum_{q_1q_2= q} \lambda_{q_1}(\delta) \mu(q_2)^2.
\end{equation*}
Note that for a fixed $\delta$, $\lambda_q$ and $\rho_q$ are multiplicative functions 
of $q$.   By counting carefully the solutions to $x^2 \equiv \delta \pmod{4q}$, 
we may check the following explicit description of the function $\lambda_q(\delta)$ 
on prime powers $q$:  Let $\delta =D l^2$ with $D$ a fundamental discriminant, 
and let $p^r$ be the exact power of $p$ dividing $l$.  Set $a=\min([k/2], r)$ and then 
we have 
\begin{equation}
\label{lam}
\lambda_{p^k}(\delta) = p^a \chi_{\delta p^{-2a}}(p^{k-2a}).
\end{equation}

We now define 
\begin{equation}
\label{eq:Loriginal}
\zL{s}{\delta} =  \frac{\zeta(2s)}{\zeta(s)} \sum_{q=1}^{\infty} \rho_q(\delta) q^{-s} = \sum_{q=1}^{\infty} \lambda_q(\delta) q^{-s}.  
\end{equation}
Note that if $\delta=0$ we have $\zL{s}{\delta} = \zeta(2s-1)$.  If $\delta$ is a nonzero discriminant, 
we may write $\delta = D l^2$ with $D$ a fundamental discriminant and then, as we 
may check using \eqref{lam}, 
\begin{equation}
\label{eq:LwithT}
\zL{s}{\delta} = l^{\half-s} T_l^{(D)}(s) L(s,\chi_D),
\end{equation}
where
\begin{equation}
\label{eq:LwithT2}
T_l^{(D)}(s) = \sum_{l_1 l_2 = l} \chi_D(l_1) \frac{\mu(l_1)}{\sqrt{l_1}} \tau_s(l_2), \qquad \tau_s(k) = k^{s-\half} \sum_{a | k} a^{1-2s}.
\end{equation}

The series $L(s,\delta)$ arose naturally in Zagier's work \cite{Zagier} 
as follows:  Consider 
\begin{equation*}
 \zeta(s,\delta) = \sum_{\{Q_{a,b,c}\}} \sum_{\substack{(m,n) \in \mz^{2}/\text{Aut}(Q_{a,b,c}) \\ Q_{a,b,c}(m,n) > 0} } \frac{1}{Q_{a,b,c}(m,n)^s},
\end{equation*}
where the outer sum is over equivalence classes of binary quadratic forms of discriminant $\delta$, and the inner sum is over equivalence classes of pairs of integers modulo the group of automorphs of the form $Q_{a,b,c}$.  Then $\zeta(s,\delta) = \zeta(s) L(s, \delta)$.   The 
expressions \eqref{eq:LwithT} and \eqref{eq:LwithT2} and the 
first two assertions of Lemma \ref{Llemma} below appear in Proposition 3 of Zagier's 
article.

\begin{mylemma} 
\label{Llemma}  Suppose 
$\delta$ is a discriminant and $\delta = Dl^2$
with $D$ a fundamental discriminant.  Then
\begin{equation}
\label{eq:twoLs}
\sum_{df^2 = \delta} L(s, \chi_d) f^{1-2s} = \zL{s}{\delta}.
\end{equation}
Furthermore, letting $\mathfrak{a} = 0$ or $1$ according to whether $\chi_D(-1) = 1$ or $-1$, respectively, we have 
\begin{equation*}
\Lambda(s,\delta):=\leg{|\delta|}{\pi}^{s/2} \Gamma\leg{s + \mathfrak{a}}{2} \zL{s}{\delta} 
= \Lambda(1-s, \delta).
\end{equation*}
Finally the zeros of $T_{l}^{(D)}(s)$ all lie on the line Re$(s)=1/2$, 
so that the Generalized Riemann Hypothesis for $\zL{s}{\delta}$ 
is equivalent to GRH for $L(s,\chi_D)$.  
\end{mylemma}
\begin{proof}
The functional equation follows upon using the functional 
equation for $L(s,\chi_D)$ together with 
$$ 
\tau_s(k) = k^{s-\frac 12} \sum_{a|k} a^{1-2s} = \sum_{ab=k} \Big(\frac ab\Big)^{\frac 12-s} 
= \tau_{1-s}(k). 
$$ 
 
Next we prove \eqref{eq:twoLs}.  If $d=Dr^2$ then $L(s,\chi_d) = L(s,\chi_D) \sum_{u|r} 
\mu(u) \chi_D(u)u^{-s}$.  Thus the left hand 
side of \eqref{eq:twoLs} is 
\begin{equation*}
L(s,\chi_D) \sum_{fr = l} f^{1-2s} \sum_{u|r} \frac{\mu(u)\chi_D(u)}{u^s} 
 = L(s,\chi_D) \sum_{u| l} \frac{\mu(u)\chi_D(u)}{u^s} \sum_{f| (l/u)} f^{1-2s}.  
\end{equation*} 
We easily recognize the RHS above as $L(s,\delta)$, as 
claimed.

It remains lastly to show that the zeros of $T_{l}^{(D)}(s)$ lie 
on the critical line. 
Since $T_{l}^{(D)}(s)$ is a multiplicative function of $l$, we need only 
consider $T_{p^k}^{(D)}(s)$.  Note that $T_{p^k}^{D}(s) = \tau_s(p^k) -\tau_s(p^{k-1})  
\chi_D(p)/\sqrt{p}$.  Letting $Z= p^{1/2-s}$ we find that $\tau_s(p^j) = (Z^j-Z^{-j})/(Z-Z^{-1})$.  
Setting also $\epsilon = \chi_D(p)/\sqrt{p}$ we obtain that
\begin{equation}
\label{eq:TZ}
 T_{p^k}^{(D)}(s) = \frac{Z^{k+1}-\epsilon Z^k + \epsilon Z^{-k} -Z^{-k-1}}{Z-Z^{-1}}.
\end{equation}
Let $p_{\epsilon}(Z)$ denote the 
numerator above.  Note that $p_{\epsilon}(\pm 1)=0$, and 
that it has an additional $2k$ zeros for $Z \in \mc$.
With the substitution $Z = e^{i \theta}$, 
\begin{equation*}
(2i)^{-1} p_{\epsilon}(e^{i \theta}) = \sin((k+1)\theta) - \epsilon \sin(k \theta).
\end{equation*}
Observe that $\sin((k+1)\theta)/\sin(k\theta)$ has singularities at $\theta=\pi j/k$ for $j=0,1,\dots (2k-1)$ and in each interval $(\pi j/k, \pi (j+1)/k)$ takes every real value exactly once.
Thus $p_{\epsilon}(e^{i \theta})$ has $2k$ zeros in $(0, 2 \pi)$, and so 
all the zeros of $p_{\epsilon}(Z)$ are on the unit circle.  It 
follows that all the zeros of $T_{p^k}^{(D)}(s)$ are on the critical line.
\end{proof}

We now return to our discussion of prime geodesics.  
Taking $s=1$ in \eqref{eq:twoLs} we find that 
$$ 
\sqrt{\delta} L(1,\delta) = \sum_{df^2 =\delta} \sqrt{d} L(1,\chi_d). 
$$ 
Using this in \eqref{eq:L1} we arrive at the following 
Proposition. 

\begin{myprop}
\label{coro:L}
Recall that $X=\sqrt{x}+1/\sqrt{x}$.  Then we have 
\begin{equation}
\label{eq:psi1}
\Psi_{\Gamma}(x) = 2 \sum_{n \leq X} \sqrt{n^2-4} \zL{1}{n^2-4}.
\end{equation}
\end{myprop}

This expression for $\Psi_{\Gamma}(x)$ may be found in Bykovskii (see \cite{B}, (2.2)) 
who quotes a preprint of Kuznetsov \cite{Kuznetsov} which is difficult to find.
 
 For $n >2$, the sum 
\begin{equation*}
\label{eq:hsumindef}
\sum_{df^2 = n^2-4} h(d) \log \epsilon_d = \sum_{df^2 = n^2-4} \sqrt{d} L(1, \chi_d)
\end{equation*}
appearing in \eqref{eq:L1} is reminiscent of the formula for the Hurwitz class number 
for negative discriminants.   Furthermore, Zagier (\cite{Zagier}, Proposition 3(iv)) observed that for 
negative discriminants $\delta$, one has that $\sqrt{|\delta|}L(1,\delta)/\pi$ 
equals the Hurwitz class number $H(\delta)$.   The Hurwitz class number 
for positive discriminants appears less well-known, but we refer to a 
paper of McKee \cite{McKee} who gives an analogous definition of $H(\delta)$  
for positive discriminants.  Rather nicely, it turns out that the 
``Hurwitz class number formula" $\sqrt{\delta} L(1,\delta)=H(\delta) \log \epsilon_\delta$ 
holds for positive discriminants.  

Our results are based upon analyzing the $L$-values appearing 
in Proposition \ref{coro:L}.   To this end, we derive a useful relation 
connecting $\lambda_q(n^2-4)$ and $\rho_q(n^2-4)$ to Kloosterman sums.

\begin{mylemma} 
 \label{lamformula} 
For any natural number $q$ and $n\ge 3$ we have 
\[ 
\rho_q(n^2-4) = \frac{1}{q} \sum_{k \shortmod{q} } e\Big(\frac{kn}{q}\Big) S(k,k;q), 
\]
and 
\[
 \lam_q(n^2-4)= \sum_{q_1^2 q_2=q} \frac{1}{q_2} \sum_{k \shortmod {q_2}} 
e\Big(\frac{kn}{q_2}\Big) S(k^2,1;q_2).
\]
If we write $q=a^2b$ with $b$ square-free, then for any $z\ge 2$ we have 
$$ 
\sum_{n\le z} \lambda_q(n^2-4) = z \frac{\mu(b)}{b} + O(q^{\frac 12+\epsilon}). 
$$
 \end{mylemma}
\begin{proof}
There is a one-to-one correspondence between solutions $y \pmod{2q}$ to the congruence $y^2 \equiv n^2 -4 \pmod{4q}$ and solutions $y_1 \pmod{q}$ to $y_1^2 + ny_1 + 1 \equiv 0 \pmod{q}$.
Notice that any such solution necessarily has $y_1$ coprime to $q$.  Thus, 
using the orthogonality of additive characters, we have 
\begin{align*}
 \rho_{q}(n^2-4) = \sumstar_{\substack {y_1 \shortmod{q} \\ n\equiv -y_1 -\overline{y_1}\pmod q}} 1 
= \frac{1}{q} \sum_{k \shortmod{q}} e\Big(\frac{kn}{q} \Big)S(k,k;q), 
\end{align*}
proving our first formula.  

Recall that  $\lambda_q(n^2-4) =  \sum_{q_1^2 q_2 q_3 =q} \mu(q_2)\rho_{q_3}(n^2-4)$.  
Using the formula just established, and Selberg's formula $S(z, z;q_3) =
 \sum_{d | (z,q_3)} d S(z^2/d^2, 1;q_3/d)$ 
 we obtain 
\begin{equation*}
 \lambda_q(n^2-4) = \sum_{q_1^2 q_2 q_3  = q} \frac{\mu(q_2)}{q_3} \sum_{k \shortmod{q_3}} e\Big(
 \frac{kn}{q_3} \Big) \sum_{d|(q_3,k)} d S(k^2/d^2,1;q_3/d).
\end{equation*}
 Calling $q_3/d$ as $q_4$ we may rewrite the above as 
 $$ 
 \sum_{q_1^2 q_4|q} \frac{1}{q_4} \sum_{k\shortmod {q_4}} e\Big(\frac{kn}{q_4}\Big) 
 S(k^2,q;q_4) \sum_{q_2|(q/q_1^2 q_4)} \mu(q_2) 
 = \sum_{q_1^2 q_4 =q} \frac{1}{q_4}  \sum_{k\shortmod {q_4}} e\Big(\frac{kn}{q_4}\Big) 
 S(k^2,q;q_4),
$$
proving our second formula.

From the formula for $\lambda_q(n^2-4)$ that was just established, we see that 
$$ 
\sum_{n\le z} \lambda_q(n^2-4) = 
\sum_{q_1^2 q_2=q} \frac 1{q_2} \sum_{k \shortmod{q_2}} S(k^2,1;q_2) \sum_{n\le z} e(kn/q_2).
$$ 
When $k=0$ the inner sum over $n$ is $z+O(1)$, and $S(0,1;q_2)= \mu(q_2)$, so 
that this term gives the stated main term  $z\mu(b)/b + O(1)$.   When $k\neq 0$ 
we use the Weil bound $S(k^2,1;q_2) \ll q_2^{\frac 12+\epsilon}$ together with the 
estimate $\sum_{n\le z} e(kn/q_2) \ll \Vert k/q_2 \Vert^{-1}$, where $\Vert \cdot \Vert$ 
denotes the distance to the nearest integer.  From these estimates we 
obtain readily that the error term from $k\neq 0$ terms is $O(q^{\frac 12+\epsilon})$.  
\end{proof}

We end this section by quickly sketching the proof of Theorem \ref{thm:L1sum}. 

\begin{proof}[Proof of Theorem \ref{thm:L1sum}]  By our description of prime geodesics 
we see that, since $w$ is even and vanishes on $[-2,2]$,  
$$ 
\sum_{\{P\} } \Lambda(P) \frac{w({\rm tr}(P)}{\sqrt{{\rm tr}(P)^2-4}} = 
\sum_{n=3}^{\infty} 2w(n) L(1,n^2-4) 
= \lim_{Q\to \infty} \sum_{q\le Q} \frac 1q \sum_{n\in \mathbb{Z}} \lam_q(n^2-4) w(n).   
$$ 
Using Lemma \ref{lamformula} we  find that the inner sum over $n$ above equals 
$$
\sum_{q_1^2 q_2 =q} \frac{1}{q_2} \sum_{k \shortmod{q_2}} S(k^2, 1;q_2) \sum_{n} w(n) e(kn/q_2),
$$ 
which by using Poisson summation may be written as 
$$ 
\sum_{q_1^2 q_2 =q} \frac{1}{q_2} \sum_{k \shortmod{q_2}} S(k^2,1;q_2) 
\sum_{\ell \in \mathbb{Z}} {\widehat w}(\ell+k/q_2) 
= \sum_{q_1^2 q_2 =q} \frac{1}{q_2} \sum_{\ell \in \mathbb{Z}} S(\ell^2, 1;q_2) {\widehat w}(\ell/q_2).
$$
This yields Theorem \ref{thm:L1sum}.  
\end{proof}
 
 \section{Deducing Theorem \ref{thm:mainthm} from auxiliary results}
\label{section:Smoothing}

\noindent As mentioned in the Introduction, our proof of Theorem \ref{thm:mainthm} 
hinges on two different ways of counting prime geodesics.  
One is via the Selberg trace formula and the other is via the connection with $L$-values described in Lemma \ref{coro:L}.   We describe here the results arising from each approach, and 
deduce Theorem \ref{thm:mainthm}, while the  proofs of the auxiliary results 
will be given in Sections 4 and 5.  

First we describe the result from the trace formula where we shall count 
prime geodesics with a certain smooth weighting.   Let $x^{\half + \varepsilon} \leq Y \leq x/\log{x}$ be a parameter to be chosen later.  Let $k(u)$ be a smooth, real-valued function with compact support 
on $(0, Y)$.  We suppose that $\intR k(u) du = \int_0^{Y} k(u) du = 1$ and that for 
all $j\ge 0$, 
\begin{equation}
\label{eq:k'}
\intR |k^{(j)}(u)| du \ll_j Y^{-j}.
\end{equation}
We consider a smoothed version of $\Psi_{\Gamma}(x)$; namely,
\begin{equation}
\label{eq:psik}
\Psi_{\Gamma}(x;k) = \int_0^{Y} \Psi_{\Gamma}(x+u) k(u) du.
\end{equation}  
 
 \begin{mytheo}
\label{thm:smooth}
With notation as above, we have
\begin{equation}
\label{eq:smooth}
\Psi_{\Gamma}(x;k) = x + \int_0^{Y} u k(u) du + E(x;k),
\end{equation}
where $E(x;k)$ satisfies
\begin{equation}
\label{eq:EkSelberg}
 E(x;k) \ll x^{3/2+\varepsilon} Y^{-1}, \text{ and}
\end{equation}
\begin{equation}
\label{eq:EkLuoSarnak}
 E(x;k) \ll x^{7/8 + \varepsilon} Y^{-1/4}.
\end{equation}
\end{mytheo}

Theorem \ref{thm:smooth} will be proved in the next section.  For the moment, we 
note that the estimate \eqref{eq:EkSelberg} follows immediately from the 
Selberg trace formula, but the estimate \eqref{eq:EkLuoSarnak} is more involved and 
relies on the work of Luo and Sarnak \cite{LuoSarnak}.
  
Now we also have 
\begin{equation}
\label{eq:smoothsharp}
\Psi_{\Gamma}(x) =\Psi_{\Gamma}(x) \int_0^{Y} k(u) du = \Psi_{\Gamma}(x;k) - \int_0^{Y} (\Psi_{\Gamma}(x+u)-\Psi_{\Gamma}(x)) k(u) du.
\end{equation}
The second term above which counts prime geodesics 
in short intervals will be handled using Lemma \ref{coro:L} together with 
estimates for Dirichlet $L$-functions.  A key input here is the 
work of Conrey and Iwaniec \cite{ConreyIwaniec} which bounds 
quadratic Dirichlet $L$-functions on the critical line.  
We state now the Theorem in this regard, which will be established in 
Section \ref{section:short}.

\begin{mytheo}
\label{thm:short}
Let $D$ denote a fundamental discriminant, and suppose that the bound 
\begin{equation}
\label{eq:Lbound}
L(\thalf + it, \chi_D) \ll (1 + |t|)^A |D|^{\theta + \varepsilon} 
\end{equation}
holds for some fixed $A > 0$ and a real number $\theta \ge 0$.  Then
\begin{equation}
\label{eq:short}
\Psi_{\Gamma}(x+u)-\Psi_{\Gamma}(x) = u + O(u^{\half} x^{\frac14 + \frac{\theta}{2} + \varepsilon}).
\end{equation}
\end{mytheo}

With $\theta$ as in Theorem \ref{thm:short}, we see from \eqref{eq:short}  that 
 \begin{equation*}
\label{eq:sharp2}
 \int_0^{Y} (\Psi_{\Gamma}(x+u)-\Psi_{\Gamma}(x)) k(u) du = \int_0^{Y} u k(u) du + O(Y^{\half} x^{\frac14 + \frac{\theta}{2} + \varepsilon}).
\end{equation*}
Using this with \eqref{eq:smoothsharp} and \eqref{eq:smooth} we 
deduce that 
\begin{equation} 
\label{eq:sharp3} 
\Psi_{\Gamma}(x) = x +E(x;k) +O(Y^{\frac 12} x^{\frac 14+\frac{\theta}{2}+\epsilon}).
\end{equation}

If we use the straightforward bound \eqref{eq:EkSelberg} above, we find 
that $\Psi_{\Gamma}(x) = x + O(x^{\frac 23 +\frac{\theta}{3}+\epsilon})$, 
upon choosing $Y=x^{\frac 56 -\frac{\theta}{3}}$.   
The Burgess bound $\theta = 3/16$ then gives 
$O(x^{\frac{35}{48} + \varepsilon})$, which is Iwaniec's original result, with a very different proof.  
Further, note that the Lindel\"{o}f hypothesis, which permits $\theta = 0$, gives $x^{2/3 + \varepsilon}$.

If we use the more sophisticated Luo-Sarnak bound \eqref{eq:EkLuoSarnak} in 
\eqref{eq:sharp3}, and choose $Y=x^{\frac 56 -\frac{2\theta}{3}}$, we find that $\Psi_{\Gamma}(x) = 
x+ O(x^{\frac 23 +\frac{\theta}{6}+\epsilon})$.  The work of Conrey and Iwaniec \cite{ConreyIwaniec} 
allows us to take $\theta=1/6$, and this gives Theorem \ref{thm:mainthm}.   Note that 
the assumption of the Lindel{\" o}f hypothesis does not give here 
an improvement of the $x^{\frac 23+\epsilon}$ bound. 
Iwaniec gave a heuristic for the bound $E(x) = O(x^{\frac 12+\epsilon})$ with a kind of extended Linnik-Selberg conjecture (see p.139 of \cite{IwaniecGeodesic}), but it would be interesting to find other heuristics, say using GRH.



\section{Proof of Theorem \ref{thm:smooth}}
\label{section:smooth}

\noindent Analogously to the explicit formula in prime 
number theory, Iwaniec (\cite{IwaniecGeodesic}, see Lemma 1) 
showed that for $1\le T \le \sqrt{x}/(\log x)^2$  
\begin{equation}
\label{eq:explicit}
 \Psi_{\Gamma}(x) = x + \sum_{|t_j| \leq T} \frac{x^{s_j}}{s_j} + O\left(\frac{x}{T} \log^2{x}\right),
\end{equation}
where $s_j = 1/2+it_j$ runs over the zeros of $Z(s)$ on the line Re$(s)=1/2$, and 
the zeros are counted with multiplicity.   
From this, it follows that 
$$ 
\Psi_{\Gamma}(x;k) = \int_{0}^{Y} k(u) \Big( x+u + \sum_{|t_j|\le T} \frac{(x+u)^{s_j}}{s_j} 
+ O\Big(\frac{x}{T}\log^2 x \Big)\Big) du. 
$$ 
Thus, choosing $T=\sqrt{x}/(\log x)^3 $, we conclude that 
\begin{equation} 
\label{eq:exp1} 
E(x;k) = \sum_{|t_j| \le \sqrt{x}/(\log x)^3} \frac{1}{s_j} \int_0^Y (x+u)^{s_j} k(u)du + O(x^{\frac 12+\epsilon}). 
\end{equation}

Integrating by parts $\ell$ times and using \eqref{eq:k'} we find that 
$$ 
\int_{0}^Y (x+u)^{s_j} k(u) du = (-1)^{\ell} \int_0^Y \frac{(x+u)^{s_j+\ell}}{(s_j+1)\cdots(s_j+\ell)} 
k^{(\ell)}(u) du \ll_{\ell } \frac{x^{\frac 12+\ell}}{(|s_j| Y)^{\ell}}. 
$$ 
Choosing $\ell$ suitably large, and recalling that there 
are $O(T^2)$ eigenvalues $s_j=1/2+it_j$ with $|t_j|\le T$, we 
find that the contribution of terms with $|t_j| \ge x^{1+\epsilon}/Y$ to \eqref{eq:exp1} 
is at most $O(x^{\frac 12+\epsilon})$.   Thus we have 
\begin{equation} 
\label{eq:exp2} 
 E(x;k) = \sum_{|t_j| \le x^{1+\epsilon}/Y} \frac{1}{s_j} \int_0^Y (x+u)^{s_j} k(u)du + O(x^{\frac 12+\epsilon}). 
\end{equation}

Clearly the sum above is $\ll x^{\frac 12} \sum_{|t_j|\le x^{1+\epsilon}/Y} 1/|s_j| 
\ll x^{\frac 32+\epsilon}/Y$, by using again that there are $O(T^2)$ eigenvalues 
with $|t_j|\le T$.  This yields the bound \eqref{eq:EkSelberg}.

To obtain \eqref{eq:EkLuoSarnak}, we invoke 
(58) of Luo and Sarnak \cite{LuoSarnak}, which gives for $v\geq 2$
\begin{equation*}
\label{eq:LS}
\sum_{|t_j| \leq T} v^{i t_j} \ll T^{5/4} v^{1/8} \log^2{T}.
\end{equation*}
 Using this and partial summation in \eqref{eq:exp2}, we obtain \eqref{eq:EkLuoSarnak}.

\section{The short interval result: Proof of Theorem \ref{thm:short}}
\label{section:short}

\noindent Recall that $X =\sqrt{x}+1/\sqrt{x}$ and we now set $X^{\prime}=\sqrt{x+u}+1/\sqrt{x+u}$.  
We start with Lemma \ref{coro:L}, which gives
\begin{equation*}
\label{eq:shortstep1}
\Psi_{\Gamma}(x+u) -\Psi_{\Gamma}(x) = 2 \sum_{X < n \leq X'} \sqrt{n^2-4} \zL{1}{n^2-4} 
= \Big(2+O\Big(\frac 1x\Big)\Big) \sum_{X<n\le X^{\prime}} n L(1,n^2 -4), 
\end{equation*}
since $\sqrt{n^2-4} = n(1+O(1/n^2))$.

Next we use a standard technique to approximate $\zL{1}{n^2-4}$ by a suitable Dirichlet series.  
Let $V \ge 1$ be a parameter to be chosen shortly, and write $\delta =n^2-4 =D l^2$ 
with $D$ denoting a fundamental discriminant.   Consider 
$$ 
S_V(\delta) = \sum_{q=1}^{\infty} \frac{\lambda_q(\delta)}{q} e^{-q/V} 
= \frac{1}{2\pi i} \int_{(1)} L(1+s,\delta) V^s \Gamma(s)ds.
$$ 
We now move the line of integration to Re$(s)=-\frac 12$, and cross a pole at $s=0$.      
Thus 
$$ 
S_V(\delta) = L(1,\delta) + \frac{1}{2\pi i} \int_{(-\frac 12)} L(1+s,\delta) V^s \Gamma(s) ds,
$$ 
and using \eqref{eq:Lbound} we obtain that the integral above is 
\begin{equation*}
\ll V^{-\half} n^{\varepsilon} \intR |\Gamma(\thalf + it)| |L(\thalf + it, \chi_D)| dt
\ll n^{2\theta + \varepsilon} V^{-\half}.
\end{equation*}
Thus we conclude that 
\begin{equation}
\label{eq:shortstep2}
\Psi_{\Gamma}(x+u) -\Psi_{\Gamma}(x) 
= \Big(2 +O\Big(\frac 1x\Big)\Big) \sum_{X < n \leq X'}  nS_V(n^2-4) 
 + O(u V^{-\half} X^{2 \theta + \varepsilon}).
\end{equation}
 
 If $q=a^2b$ with $b$ square-free, then we find from Lemma \ref{lamformula} and partial 
 summation
\begin{equation*}
\label{eq:lambdaavg}
 2\sum_{X < n \leq X'} n\lambda_q(n^2-4) 
= (u+O(X))  \frac{\mu(b)}{b} + O(Xq^{\half + \varepsilon}).
\end{equation*}
Therefore, 
\begin{equation*} 
 2\sum_{X< n\le X'} n S_V(n^2-4)= (u+O(X)) \sum_{a, b} 
\frac{\mu(b)}{a^2 b^2} e^{-a^2b/V} + O(X V^{\frac 12+
\varepsilon}). 
\end{equation*}
Now, by a standard contour shift argument,  
\begin{equation*}
\sum_{a, b} \frac{e^{-a^2 b/V} \mu(b)}{a^2 b^2} = 
 \frac{1}{2 \pi i} \int_{(1)} V^{s} \Gamma(s) \frac{\zeta(2+2s)}{\zeta(2+s)} ds
= 1+ O(V^{-\half}).
\end{equation*}
Using the above remarks in \eqref{eq:shortstep2} we 
conclude that 
\begin{equation*}
\Psi_{\Gamma}(x+u) - \Psi_{\Gamma}(x) = u  + O(X V^{\half + \varepsilon} 
 + u V^{-\half} X^{2 \theta + \varepsilon}).
\end{equation*}
The optimal choice for $V$ is $V = u X^{-1 + 2\theta}$ which gives Theorem \ref{thm:short}.

\section{Very short intervals: Proof of Theorem \ref{thm:shortintro}}

\noindent If $h\ge x^{25/36+\epsilon}$ then Theorem \ref{thm:shortintro}  
holds unconditionally by Theorem \ref{thm:mainthm}.  We suppose 
below that $h\le x^{25/36+\epsilon}$, and put $X=\sqrt{x}+1/\sqrt{x}$ 
and $X+\Delta = \sqrt{x+h}+1/\sqrt{x+h}$ so that $\Delta \sim h/(2\sqrt{x})$.  
Using \eqref{eq:psi1}, we find that 
\begin{equation*} 
 \Psi_{\Gamma}(x+h) - \Psi_{\Gamma}(x)= \sum_{X<n\le X+\Delta} 2\sqrt{n^2-4} L(1,n^2-4) 
 \sim 2\sqrt{x} \sum_{X\le n\le X+\Delta} L(1,n^2-4).   
\end{equation*}
 Thus, to establish Theorem \ref{thm:shortintro} we 
 need only prove that, on GRH, if  
  $(\log{X})^{2 + \varepsilon} \leq \Delta \leq X^{\varepsilon}$ then
\begin{equation}
\label{eq:Lshort}
 \sum_{X < n \leq X + \Delta} \zL{1}{n^2-4} \sim \Delta.
\end{equation}
  
 Let $\delta\le 2X^2$ be a discriminant, and write as before $\delta= D l^2$ 
 with $D$ being fundamental.  We set below $Z= (\log X)^2 (\log \log X)^8$.  Then we find that on GRH 
\begin{equation*}
 L(1,\chi_D) = \prod_{p\le Z} \Big(1-\frac{\chi_D(p)}{p}\Big)^{-1} 
\Big(1 + O\Big(\frac{1}{(\log \log X)^2} \Big)\Big).  
\end{equation*}
The above estimate is standard and versions of it go back to Littlewood; for example one 
may deduce it from Lemma 2.1 of \cite{GS}.  
Note also that 
 $$ 
 l^{-1/2} T_l^{(D)}(1) =\prod_{p^a \Vert l} \Big(1+\sum_{k=1}^{a} \frac{1-\chi_D(p)}{p^k}\Big).
 $$ 
 The contribution of the primes $p>Z$ to the above product 
 is clearly bounded by 
 $$
 \exp\Big(\sum_{p|l, \ p>Z} \frac 2p\Big) = \exp\Big( O\Big(\frac{1}{Z}\sum_{p|l} 1\Big)\Big) 
 = \exp\Big(O\Big(\frac{\log X}{Z}\Big)\Big).
 $$  
 Thus
 $$ 
 L(1,\delta) =\prod_{p\le Z}\Big( \sum_{k=0}^{\infty} \frac{\lambda_{p^k}(\delta)}{p^k} \Big) 
 \Big( 1+ O\Big( \frac{1}{(\log \log X)^2}\Big).
  $$ 
 Using \eqref{lam}, 
 we note that the product above is $\ll \log Z \ll \log \log X$.  Moreover, setting 
 $z= (\log \log X)^{2-\epsilon}$, and 
 $$
 S_1(\delta) = \sum_{z<p \le Z} \frac{\lambda_p(\delta)}{p}, \qquad 
 \text{and} \qquad S_2(\delta) = \sum_{\substack{z<p\le Z \\ p|l}} 
 \frac{1}{p} 
 $$ 
 we note, using \eqref{lam}, that the product over the primes in $[z,Z]$ is 
$$
\prod_{z\le p\le Z} \Big(1 +\frac{\lambda_p(\delta)}{p} + \frac{\lambda_{p^2}(\delta)}{p^2} 
+ O\Big(\frac{1}{p^2}\Big) \Big)
= \exp(S_1(\delta)+S_2(\delta))(1+O(1/z)).
$$
 We thus conclude that 
 \begin{equation} 
 \label{L1}
 L(1,\delta) = \prod_{p\le z} \Big(\sum_{k=0}^{\infty} \frac{\lambda_{p^k}(\delta)}{p^k} 
 \Big) \exp(S_1(\delta)+S_2(\delta) ) + O\Big( \frac{1}{(\log \log X)^{1-\epsilon}}\Big).
 \end{equation}
 
 Let ${\mathcal B}$ denote the set of bad $n \in [X,X+\Delta]$ for which 
 $|S_1(n^2-4)|\ge 1$, or $S_2(n^2-4)\ge 1$.  For $n$ which are not bad we 
 may use that $\exp(S_1(n^2-4)+S_2(n^2-4)) = 1+ O(S_1(n^2-4)+S_2(n^2-4))$.  
 Thus using \eqref{L1} we find that 
 \begin{align*}
 \sum_{X< n\le X+\Delta} L(1,n^2-4)&= 
 \sum_{X<n\le X+\Delta} \prod_{p\le z}\Big(\sum_{k=0}^{\infty} \frac{\lambda_{p^k}(n^2-4)}{p^k} 
 \Big) \Big(1+O(S_1(n^2-4)+S_2(n^2-4))\Big) 
 \\
 &+ O(|{\mathcal B}| \log \log X ) +o(\Delta).
 \end{align*}
 Now note that 
 $$ 
 \sum_{X\le  n\le X+\Delta} S_2(n^2-4) \le \sum_{z<p\le Z} \frac{1}{p} 
 \sum_{\substack{X\le n\le X+\Delta\\ p^2 |n^2-4}} 1 \ll \frac{\Delta}{z}.
 $$ 
 From this we see that the set of $n \in [X,X+\Delta]$ with $S_2(n^2-4) \ge 1$ is 
of size $\ll \Delta/z$, and moreover 
$$ 
\sum_{X\le n\le X+\Delta} \prod_{p\le z}\Big( \sum_{k=0}^{\infty} 
\frac{\lambda_{p^k}(n^2-4)}{p^k} \Big) S_2(n^2-4) \ll (\log z) \frac{\Delta}{z} = o(\Delta).
$$ 

 We shall establish that 
 \begin{equation} 
 \label{UB1} 
 \sum_{X\le n\le X+\Delta}   S_1(n^2-4)^2  
 \ll \frac{\Delta}{z} + Z^{1+\epsilon},
 \end{equation} 
 \begin{equation}
 \label{UB2} 
 \sum_{X\le n\le X+\Delta} 
 \prod_{p\le z} \Big(\sum_{k=0}^{\infty} \frac{\lambda_{p^k}(n^2-4)}{p^k} 
 \Big) S_1(n^2-4)^2  = o(\Delta),
 \end{equation}
 and 
 \begin{equation} 
 \label{E1} 
 \sum_{X\le n \le X+\Delta} 
 \prod_{p\le z} \Big(\sum_{k=0}^{\infty} \frac{\lambda_{p^k}(n^2-4)}{p^k} 
 \Big) = \Delta +o(\Delta) . 
 \end{equation} 
 From estimate \eqref{UB1} we find that the set of $n$ with $|S_1(n^2-4)| \ge 1$ 
 has size $\ll \Delta/z$.  Thus $|{\mathcal B}| \ll \Delta/z$ and so the term 
 $O(|{\mathcal B}|\log \log X)$ is $o(\Delta)$.  
 By Cauchy's inequality and \eqref{UB2} and \eqref{E1} we 
 find that 
 $$ 
 \sum_{X\le n\le X+\Delta} \prod_{p\le z}\Big(\sum_{k=0}^{\infty} \frac{\lambda_{p^k}(n^2-4)}{p^k} 
 \Big) |S_1(n^2-4)|  =o(\Delta).
 $$ 
 Thus once \eqref{UB1}, \eqref{UB2} and \eqref{E1} are established, \eqref{eq:Lshort}
 would follow.

   Let ${\mathcal S}(z)$ denote the set of all integers $q$ composed only of prime factors 
 below $z$.  Then we have, writing $q=a^2 b$ with $b$ square-free and using Lemma 
 \ref{lamformula},
 $$ 
 \sum_{X\le n\le X+\Delta} \sum_{q\in {\mathcal S}(z)} \frac{\lambda_q(n^2-4)}{q} 
 = \sum_{q=a^2 b\in {\mathcal S}(z)} \Big( \Delta \frac{\mu(b)}{a^2 b^2} 
 + O\Big( \frac{1}{(a^2 b)^{\frac 12-\epsilon}}\Big) \Big)
 =\Delta + O\Big( \exp\Big( \sum_{p\le z}\frac{1}{p^{\frac12-\epsilon}}\Big)\Big).
 $$ 
 Since $z=(\log \log X)^{2-\epsilon}$ the error term 
 above is $o(\Delta)$ and \eqref{E1} follows.  
 
 Next note that 
 $$ 
 \sum_{X\le n \le X+\Delta} 
 \Big| \sum_{z< p\le Z} \frac{\lambda_p(n^2-4)}{p }\Big|^2 
 = \sum_{z< p_1, p_2 \le Z } \frac{1}{p_1 p_2} \sum_{X\le n\le X+\Delta} \lambda_{p_1}(n^2-4) 
 \lambda_{p_2}(n^2-4). 
 $$ 
 The terms $p_1=p_2$ contribute $\ll \Delta \sum_{z<p\le Z} 1/p^2 \ll \Delta/z$.  
 As for the terms $p_1\neq p_2$, these contribute 
 $$ 
 \ll \sum_{z<p_1 \neq p_2\le Z} \frac{1}{p_1 p_2} \Big( \frac{\Delta}{p_1p_2} + O((p_1p_2)^{\frac 12+\epsilon})\Big) 
 \ll \frac{\Delta}{z^2} + Z^{1+\epsilon}. 
 $$ 
 Thus \eqref{UB1} has been established.  
 
 Finally note that the quantity to be estimated in \eqref{UB2} is 
 $$ 
  \sum_{X\le n\le X+\Delta} 
\Big( \sum_{q\in {\mathcal S}(z)} \frac{\lambda_q(n^2-4)}{q}\Big) \sum_{z<p_1, p_2 \le Z} 
 \frac{\lambda_{p_1}(n^2-4)\lambda_{p_2}(n^2-4)}{p_1 p_2}. 
 $$ 
 The terms $p_1 =p_2$ contribute an amount 
 $$ 
 \ll \sum_{z<p \le Z} \frac{1}{p^2} \sum_{X\le n\le X+\Delta} \prod_{p\le z} \Big(\sum_{k=0}^{\infty} 
 \frac{\lambda_{p^k}(n^2-4)}{p^k} \Big) \ll \frac{\Delta}{z}, 
 $$ 
 upon using \eqref{E1}.  As for the terms $p_1 \neq p_2$, note that $q\in {\mathcal S}(z)$ 
 is coprime to $p_1p_2$ and so such terms contribute (with $q=a^2 b$ and $b$ square-free)
 $$ 
 \ll \sum_{q=a^2 b \in {\mathcal S}(z)} \frac{1}{a^2 b} 
 \sum_{z< p_1 \neq p_2 \le Z} \frac{1}{p_1 p_2} \Big( \frac{\Delta}{bp_1p_2} + (qp_1p_2)^{\frac 12+\epsilon}\Big) 
 \ll \frac{\Delta}{z^2} + Z^{1+\epsilon} \exp\Big(\sum_{p\le z} \frac{1}{p^{\frac 12-\epsilon}}\Big). 
 $$ 
 Since $\Delta \ge (\log X)^{2+\epsilon}$ and $z= (\log \log X)^{2-\epsilon}$ 
 this is $o(\Delta)$, proving \eqref{UB2} and hence also \eqref{eq:Lshort} 
 and Theorem \ref{thm:shortintro}.

\end{document}